\documentclass[reqno,12pt]{amsart}

\usepackage{amsfonts}
\usepackage{amsmath}
\usepackage{amssymb}

\newtheorem{theorem}{Theorem}

\newtheorem{corollary}[theorem]{Corollary}

\newtheorem{definition}[theorem]{Definition}

\newtheorem{lemma}[theorem]{Lemma}

\newtheorem{proposition}[theorem]{Proposition}
\newtheorem{remark}[theorem]{Remark}

\newcommand\bld[1]{\boldsymbol{#1}}
\def\bE{\mathbb{E}}

\numberwithin{theorem}{section}
\numberwithin{equation}{section}

\makeatletter
\@namedef{subjclassname@2020}{\textup{2020} Mathematics Subject Classification}
\makeatother

\begin{document}
\title[Stationary NSE]
{Intrusive and Non-Intrusive Polynomial Chaos Approximations for a Two-Dimensional Steady State Navier-Stokes System with Random Forcing}

\author{S. V. Lototsky}\thanks{SVL: Research supported by ARO Grants W911N-16-1-0103 and  FA9550-21-1-0015}
\curraddr[S. V. Lototsky]{Department of Mathematics, USC\\
Los Angeles, CA 90089}
\email[S. V. Lototsky]{lototsky@usc.edu}
\urladdr{https://dornsife.usc.edu/sergey-lototsky/}

\author{R. Mikulevicius}
\curraddr[R. Mikulevicius]{Department of Mathematics, USC\\
Los Angeles, CA 90089}
\email[R. Mikulevicius]{mikulvcs@usc.edu}
\urladdr{https://dornsife.usc.edu/cf/faculty-and-staff/faculty.cfm?pid=1003536}

\author{B. L. Rozovsky}\thanks{RM and BLR: Research supported by ARO Grant W911N-16-1-0103}
\curraddr[B. L. Rozovsky]{Division of Applied Mathematics\\
Brown University\\
Providence, RI 02912}
\email[B. L. Rozovsky]{Boris\_Rozovsky@Brown.edu}
\urladdr{www.dam.brown.edu/people/rozovsky/rozovsky.htm}



\begin{abstract}

While convergence of polynomial chaos approximation
for linear equations is relatively well understood, a lot less is known
for non-linear equations. The paper investigates this convergence
 for a particular equation with  quadratic  nonlinearity.

\end{abstract}

\maketitle

\today

\section{Introduction}
There are two main ways to study an
 equation with random input. One way is to use deterministic tools for
 each particular realization of randomness; in what follows, we call it
 {\tt path-wise approach}. An alternative, which we call   {\tt random field approach}, is to consider the random input as an additional independent
 variable in the equation, along with space and/or time.

 Questions such as existence/uniqueness/regularity of the solution are
 often addressed with a combination of the two approaches; cf.
 \cite{KV, KZ} for ordinary differential equations and
 \cite{Kr_Lp1, mr04, mr05} for equations with partial derivatives.

 The difference between the two approaches becomes noticeable in
 numerical computations; see, for example,
 \cite{Ghanem-book, DXiu-book}.
  Path-wise approach leads to repeated
 numerical solutions of the underlying equation
 for various realizations  of the random input;
 a typical example is Monte Carlo simulations. In computational terms,
  this approach is  {\tt non-intrusive}, because  no new numerical procedures
 are required to solve the
 equation compared to the deterministic case.

  The random field approach reduces the problem to a fixed system of
 deterministic equations via a stochastic Galerkin approximation; in many cases, the result is a
  generalized polynomial chaos (gPC) expansion
 \cite[Chapter 6]{DXiu-book}.  In computational terms, this approach is   {\tt intrusive},
  because  the resulting system   is more complicated than the original
 equation and requires different numerical procedures to obtain a solution.

  The stochastic collocation method \cite[Chapter 7]{DXiu-book}, with
 sampling at  pre-determined  realizations of the random input,
 somewhat bridges the gap between pure random sampling
 (Monte Carlo) and  complete elimination of randomness
  (gPC). In computational terms, the method is {\em non-intrusive}
  \cite{Non-intrusive-gen1, Non-intrusive-gen2}. In this paper, we consider the
  {\em discrete projection}, or {\em pseudo-spectral} version of the method, when the sampled solution is
  used to approximate the coefficients in the chaos expansion via Gauss quadrature.

 For many, although apparently not all \cite{Anti-gPC}, equations,
 various empirical studies
  \cite[etc.]{HWRZ,StatStochNSE3,StatStochNSE2} suggest that
 the stochastic Galerkin approximation method, with a
 fixed computational cost, can be a much more efficient
 way to study statistical properties of the solution than  Monte Carlo
 or stochastic collocation methods.
  In the case of nonlinear equations,
  this experimental success has yet to be fully justified theoretically;
  for linear equations, the picture is rather clear; see, for example,
  \cite{BdKl, LMR, LR-AP} as well as
  \cite[Chapter 5]{LR-Book} and \cite[Section 8.3]{RL-Book}.

  Accordingly, our objective in this paper is to carry out a comparative
  theoretical analysis of  an intrusive  and  a non-intrusive  approximations for a particular nonlinear equation.
  Specifically, we consider the stationary Navier-Stokes system in a
  smooth bounded planar domain with zero boundary conditions and with
   randomness in the external force, and we establish {\em a priori} error bounds for both approximations. 

 The paper is organized as follows. Section \ref{sec1} describes the
 model and introduces the necessary function spaces. Section
 \ref{sec2} introduces the stochastic Galerkin approximation and gives the
 proof of convergence. Section \ref{sec:GQ}
 investigates a non-intrusive pseudo-spectral approximation.
 Section \ref{sec4} puts the results in a broader context.

Throughout the paper, $G$ is a bounded domain in
$\mathbb{R}^2$ with area $|G|$ and
sufficiently regular (e.g. locally Lipschtiz) boundary $\partial G$.
We use the following convention with the notations of various
function spaces and their elements: if $X$ denotes a space of scalar
 fields $f$ on $G$, then $\mathbf{X}$ denotes the corresponding
 space of vector fields $\mathbf{f}$,
 and $\mathbb{X}$ denotes the collection of
  $\mathbf{X}$-valued random elements $\bld{f}$.

\section{The Setting}
\label{sec1}
Let $\left( \Omega ,\mathcal{F},\mathbb{P}\right) $
be a probability space such that the   $L_2(\Omega)$ is a separable Hilbert space and has a complete orthogonal basis
$\{\mathfrak{P}_{n},\ n\geq 0\}$: with
\begin{equation}
\label{c(l)}
c(n)=\mathbb{E}\mathfrak{P}_{n}^2,
\end{equation}
 every $\zeta\in L_2(\Omega)$ can be represented as
  $$
\zeta=\sum_{k\geq 0}
\frac{\mathbb{E}\big(\zeta\mathfrak{P}_k)}{c(k)}
\mathfrak{P}_{k}.
$$
In what follows, we always assume that the basis $\{\mathfrak{P}_{n},\ n\geq 0\}$ has the following property:
for every $m,n\geq 0$, there are {\em finitely many} real numbers  $A_{m,n;l},\ l\geq 0,$ such that
\begin{equation}
\label{eq:Amnk}
 \mathfrak{P}_{m} \mathfrak{P}_{n}=
 \sum_{l\geq 0} A_{m,n;l} \mathfrak{P}_{l};
\end{equation}
in that case
$$
A_{m,n;l}=
\frac{\mathbb{E}\big(\mathfrak{P}_{m}\mathfrak{P}_{n}\mathfrak{P}_{l}\big)}{c(l)}.
$$
Property \eqref{eq:Amnk} holds when each $\mathfrak{P}_{n}$ is a polynomial or a tensor product of polynomials.

  Denote by  $\mathcal{P}^{N}$ the orthogonal projection
   in $L_2(\Omega)$  on the subspace spanned  by
$\left\{ \mathfrak{P}_{k},\ k=0,\ldots ,N\right\}.$

Consider  a steady-state
 Navier-Stokes system with random forcing in a  bounded
 domain $G\subset \mathbb{R}^2$ with sufficiently regular
  boundary $\partial G$:
\begin{eqnarray}
\nu \Delta \bld{u}\left( x\right) &=&
\big(\bld{u}\cdot\nabla)\bld{u} +\nabla p\left( x\right) +\bld{f}(x),\
 x\in G,  \label{10} \\
\mathrm{div}\,{}\bld{u}\left( x\right) &=&0,\
x\in G,\ \bld{u}|_{\partial G}=0.
\notag
\end{eqnarray}
In equation \eqref{10},
\begin{itemize}
\item $\nu>0$ is the kinematic viscosity coefficient,
 $x=(x_1,x_2)$,
    $\Delta=\frac{\partial^2}{\partial x_1^2}
    +\frac{\partial^2}{\partial x_2^2}$ is the Laplace operator, and $\nu$ is constant;
\item
$\bld{u}\left( x\right) =\left( u^{1}\left( x\right) ,u^{2}\left( x\right)
\right) $ is the (unknown) velocity  and
\begin{equation}
\label{operations}
\mathrm{div}\, \bld{u}=\nabla\cdot\bld{u}=
\frac{\partial u^1}{\partial x_1}+\frac{\partial u^2}{\partial x_2},\ \
\big(\bld{u}\cdot\nabla)u^i=u^1\frac{\partial u^i}{\partial x_1}+
u^2 \frac{\partial u^i}{\partial x_2},\ i=1,2;
\end{equation}
\item  $p=p\left( x\right)$ is the (unknown scalar) pressure and
$\big(\nabla p)^i=\frac{\partial p}{\partial x_i},\ i=1,2$;
\item $\bld{f}$ is the random forcing.
 \end{itemize}
 Two standard  references for the deterministic counterpart of \eqref{10}
 are  \cite[Chapter IX]{ga} and \cite[Chapter II]{temam}.

We will use the following function spaces:
\begin{itemize}
\item $\mathcal{C}_{0}^{\infty}(G)$, the collection of
infinitely differentiable real-valued
functions on $G$ with compact support in $G$;
\item $\mathcal{D}\left( G\right)
=\left\{ \bld{\varphi}=(\varphi^1,\varphi^2),\
 \varphi^i\in \mathcal{C}_{0}^{\infty}\left( G\right),\ i=1,2 :
 \mathrm{div}\,{}\bld{\varphi }=0\right\};$

\item $L^{r}\left( G\right) $, $1\leq r<+\infty$,
 the collection of measurable functions $g$ on $G$
such that
$$
\left\vert g\right\vert _{L^{r}}=\left( \int_{G}\left\vert
g(x)\right\vert ^{r}\, dx\right) ^{1/r}<\infty;
$$
for ${g}, {f}\in  L^{2}\left( G\right) $,
we write
\begin{equation*}
\left( {f,g}\right)_0 =\int_{G} f(x)g(x)\,dx;
\end{equation*}

\item  $\mathbf{L}^{r}\left( G\right)$, the collection of vector
fields $\mathbf{g}=\left(g^{1},g^{2}\right)$ on $G$ such that $g^{1},g^{2}\in L^{r}\left( G\right)$, and endowed
with norm
\begin{equation*}
\left\vert\mathbf{g}\right\vert _{\mathbf{L}^{r}}
=\Big( \left\vert g^{1}\right\vert_{L^{r}}^{r}
+\left\vert g^{2}\right\vert _{L^{r}}^{r}\Big) ^{1/r};
\end{equation*}%
for $\mathbf{g}, \mathbf{f}\in \mathbf{L}^{2}\left( G\right) $,
we write
\begin{equation*}
\left( \mathbf{f,g}\right)_0 =\int_{G}
\Big(f^{1}(x)g^{1}(x)+f^{2}(x)g^{2}(x)\Big)\,dx;
\end{equation*}

\item $\mathbb{L}^{2}\left( G\right)
    =L^2\big(\Omega;\mathbb{L}^2(G)\big)$, that is, the collection
of  $\mathbf{L}^{2}\left( G\right) $-valued  random elements
 $$
 \bld{g}(\omega,x)=
 \big( g^{1}\left(\omega,x\right), g^{2}\left(\omega, x\right) \big)
 $$
such that
\begin{equation*}
\left\vert \bld{g}\right\vert _{\mathbb{L}^{2}}
=\left( \mathbb{E}
\left\vert \bld{g}\right\vert _{\mathbf{L}^{2}}^{2}\right)^{1/2}
<\infty;
\end{equation*}

\item $H_{0}^{1,2}\left( G\right) $, the completion
 of $\mathcal{C}_{0}^{\infty}\left( G\right) $ with respect to the norm
\begin{equation*}
\left\vert g\right\vert _{1,2}
=\left( \int_{G}
\left\vert \nabla g(x)\right\vert^{2}\,dx\right) ^{1/2}=
\left( \int_{G}\left(
\left\vert \frac{\partial g(x)}{\partial x_1}\right\vert^{2}
+\left\vert \frac{\partial g(x)}{\partial x_2}\right\vert^{2}
\right)\,
dx\right) ^{1/2};
\end{equation*}%
note that $|\cdot|_{1,2}$ is  indeed a norm on $\mathcal{C}_{0}^{\infty}\left( G\right)$ because,
 by a version of the Poincar\'{e} inequality,
 if $g\in \mathcal{C}_{0}^{\infty}\left( G\right)$
  and $|G|$ is the Lebesgue measure (area) of $G$, then
  (cf. \cite[Exercise II.5.4]{ga})
\begin{equation}
\label{PI}
\left\vert g\right\vert _{L^{2}}^2
\leq \frac{\left\vert G\right\vert}{2}
\left\vert g\right\vert _{1,2}^2;
\end{equation}

\item $\mathbf{H}_{0}^{1,2}\left( G\right)$,
the collection of  vector
fields $\mathbf{g}=\left( g^{1},g^{2}\right) $ on $G$ such that $g^{1},g^{2}\in H_{0}^{1,2}\left( G\right)$, and endowed
with norm
\begin{equation*}
\left\vert \mathbf{g}\right\vert _{1,2}
=\left( \left\vert g^{1}\right\vert_{1,2}^{2}
+\left\vert g^{2}\right\vert _{1,2}^{2}\right) ^{1/2};
\end{equation*}
for $\mathbf{f}, \mathbf{g}\in \mathbf{H}^{1,2}_0(G)$,
we write
\begin{equation}
\label{trace}
\big(\nabla \mathbf{f},\nabla \mathbf{g}\big)_0=
\sum_{i,j=1}^2 \int_G \left(\frac{\partial f^i(x)}{\partial x_j}\
\frac{\partial g^i(x)}{\partial x_j}\right)\,dx,
\end{equation}
so that
$$
|\mathbf{g}|_{1,2}^2=\big(\nabla \mathbf{g},
\nabla \mathbf{g}\big);
$$

\item  $\mathbb{H}_{0}^{1,2}\left( G\right)=
    L^2\big(\Omega;\mathbf{H}_{0}^{1,2}(G)\big)$;

\item $\widehat{\mathbf{H}}^{1,2}_0(G)$, the completion of
$\mathcal{D}(G)$ with respect to the norm $|\cdot|_{1,2}$;

\item $\widehat{\mathbb{H}}^{1,2}_0(G)=
L^2\big(\Omega;\widehat{\mathbf{H}}^{1,2}_0(G)\big)$;

\item $H_{0}^{-1,2}\left( G\right)$, the completion
of $L^{2}\left( G\right) $
with respect to the norm
\begin{equation*}
\left\vert g\right\vert _{-1,2}=\sup \left\{
\int_G g(x)\varphi(x) \,dx,\
\varphi \in H_{0}^{1,2}\left( G\right),\
\vert\varphi \vert _{1,2}\leq 1,\right\};
\end{equation*}

\item $\mathbf{H}_{0}^{-1,2}\left( G\right)$, the collection of vector
fields $\mathbf{g}=\left( g^{1},g^{2}\right)$
such that $g^{1},g^{2}\in H_{0}^{-1,2}\left( G\right),$ and
endowed with norm
\begin{equation*}
\left\vert \mathbf{g}\right\vert_{-1,2}
=\left( \left\vert g^{1}\right\vert_{-1,2}^{2}
+\left\vert g^{2}\right\vert _{1,2}^{2}\right) ^{1/2};
\end{equation*}

\item  $\mathbb{H}_{0}^{-1,2}\left(G\right)
=L^2\big(\Omega;\mathbf{H}_{0}^{-1,2}(G)\big).$
\end{itemize}

The (Banach space) dual of $H_0^{1,2}\left( G\right) $ is
 isomorphic to $H_0^{-1,2}\left( G\right)$:
  see \cite[Theorem II.3.5]{ga}.
We denote the corresponding duality  by
$\left\langle f,g\right\rangle _{1}$,
 $f\in H_{0}^{-1,2}\left( g\right) ,g\in H_{0}^{1,2}\left(
G\right).$
Similarly, the dual of $\mathbf{H}^{1,2}_0\left( G\right) $ is isomorphic to $\mathbf{H}^{-1,2}_0\left( G\right) $ and the duality is denoted by $\left\langle
\mathbf{f},\mathbf{g}\right\rangle_{1}$,
$\mathbf{f}\in \mathbf{H}%
_{0}^{-1,2}\left( G\right),\
\mathbf{g}\in \mathbf{H}_{0}^{1,2}
\left(G\right).$

For $\left( \mathbf{u,v,w}\right) \in \mathbf{H}^{1,2}_0\left( G\right) \times
\mathbf{H}^{1,2}_0\left( G\right) \times \mathbf{H}_{0}^{1,2}\left( G\right)$, we define the tri-linear form
\begin{equation}
\label{form-a}
\mathfrak{a}\left( \mathbf{u,v,w}\right)
=\big(
 (\mathbf{u}\cdot\nabla)\mathbf{v},\mathbf{w}
\big)_0;
\end{equation}
similar to \eqref{operations},
$$
\big(\mathbf{u}\cdot\nabla)v^i=u^1\frac{\partial v^i}{\partial x_1}+
u^2 \frac{\partial v^i}{\partial x_2},\ i=1,2.
$$

\begin{lemma}
\label{le1} The trilinear form $\mathfrak{a}$ has the following
properties:
\begin{enumerate}
\item  If $\left( \mathbf{u,v,w}\right) \in \mathbf{H}^{1,2}_0\left( G\right) \times
\mathbf{H}^{1,2}_0\left( G\right) \times \mathbf{H}_{0}^{1,2}\left( G\right)$, then
\begin{align}
\label{cont-a0}
&\left\vert\mathfrak{a}\left( \mathbf{u,v,w}\right) \right\vert
\leq\frac{\sqrt{\left\vert G\right\vert}}{2} \,\left\vert
\mathbf{u}\right\vert_{1,2}\left\vert \mathbf{v}\right\vert _{1,2}
\left\vert \mathbf{w}\right\vert _{1,2},\\
\label{cont-a-diff}
& |\mathfrak{a}( \mathbf{u,u,w})-\mathfrak{a}( \mathbf{v,v,w})|\leq \frac{\sqrt{\left\vert G\right\vert}}{2} \,
|\mathbf{u}-\mathbf{v}|_{1,2}\big(|\mathbf{u}|_{1,2}+|\mathbf{v}|_{1,2}\big)\left\vert \mathbf{w}\right\vert _{1,2};
\end{align}%

\item If   $\mathbf{u}\in \widehat{\mathbf{H}}^{1,2}_0(G)$, then
\begin{equation}
\label{CP}
\mathfrak{a}(\mathbf{u},\mathbf{v},\mathbf{v})=0
\end{equation}
and
\begin{equation}
\label{AS}
\mathfrak{a}(\mathbf{u},\mathbf{v},\mathbf{w})=
-\mathfrak{a}(\mathbf{u},\mathbf{w},\mathbf{v}).
\end{equation}
\end{enumerate}
\end{lemma}
For the proofs, see \cite[Lemma IX.1.1]{ga} and
 \cite[Lemma IX.2.1]{ga}, respectively. Note that \eqref{cont-a0}
 follows from the H\"{o}lder inequality
 \begin{equation}
 \label{cont-a1}
 |\mathfrak{a}(\mathbf{u},\mathbf{v},\mathbf{w})|
 \leq |\mathbf{u}|_{\mathbf{L}^q}\,
 |\mathbf{v}|_{1,2}\,
 |\mathbf{w}|_{\mathbf{L}^r},
  \ \ \frac{1}{q}+\frac{1}{r}=\frac{1}{2},
  \end{equation}
by taking $q=r=4$ and using a suitable embedding theorem;
for other versions of \eqref{cont-a0} and \eqref{cont-a1},
see, for example, \cite[Exercise IX.2.1]{ga}.

In particular, with $\mathbf{w}=\mathbf{u}-\mathbf{v}$,
$$
\mathfrak{a}( \mathbf{u,u,w})-\mathfrak{a}( \mathbf{v,v,w}) =
\mathfrak{a}( \mathbf{w,u,w})+\mathfrak{a}( \mathbf{v,w,w}),
$$
so that, if $\mathbf{v}\in \widehat{\mathbf{H}}^{1,2}_0(G)$, then  \eqref{cont-a0}
and \eqref{CP} imply
\begin{equation}
\label{cont-a-diff1}
|\mathfrak{a}( \mathbf{u,u,u-v})-\mathfrak{a}( \mathbf{v,v,u-v})|\leq \frac{\sqrt{\left\vert G\right\vert}}{2} \,
|\mathbf{u}-\mathbf{v}|_{1,2}^2\,|\mathbf{u}|_{1,2}.
\end{equation}

Similar to \cite[Definition IX.1.1]{ga}, we have
\begin{definition}
\label{d1}
Let $\bld{f}\in \mathbb{H}_{0}^{-1,2}\left( G\right) $.
A random vector field
 $\bld{u}\in \widehat{\mathbb{H}}_{0}^{1,2}\left( G\right) $ is  called
 a solution to \eqref{10} if,
for every $\bld{\varphi}\in \mathcal{D}(G)$,
\begin{equation}
\label{ff1}
\mathbb{P}\Big(\,\nu\, \big(\nabla \bld{u},\nabla \bld{\varphi}\big)_0
 +
 \mathfrak{a}(\bld{u},\bld{u},\bld{\varphi})
  =
  -\left\langle \bld{f},\bld{\varphi }\right\rangle _{1}\Big) =1,
\end{equation}
where $\mathfrak{a}$ is the tri-linear form \eqref{form-a}.
\end{definition}

By applying \cite[ Lemma IX.1.2]{ga} to  \eqref{10}
with a particular realization of $\xi$, we get the following result.

\begin{lemma}
\label{le2}
If  $\bld{f}\in \mathbb{H}_{0}^{-1,2}\left( G\right)$,
 $\bld{u}\in \widehat{\mathbb{H}}_{0}^{1,2}\left(G\right)$,
  and \eqref{ff1} holds, then
   there exists a $p\in\mathbb{L}^{2}\left( G\right) $ with
$\mathbb{P}\left(\int_{G}p(x)\,dx=0\right)=1$ such that,
for every $\bld{\varphi}\in\mathbf{H}_{0}^{1,2}\left(G\right)$,
\begin{equation}
\label{ff2}
\mathbb{P}\Big(
\nu \,  \big(\nabla \bld{u},\nabla \bld{\varphi}\big)_0
+
\mathfrak{a}(\bld{u},\bld{u},\bld{\varphi})
 =
 \left( p,\mathrm{div}\,{}\bld{\varphi }\right)_0
-
\left\langle \bld{f},\bld{\varphi }\right\rangle _{1}\Big)=1.
\end{equation}
\end{lemma}

Similarly, applying \cite[Theorems IX.2.1 and IX.3.2 ]{ga} to \eqref{10},
we get the basic existence and uniqueness result.

\begin{theorem}
\label{t1}
If $\bld{f}\in \mathbb{H}_{0}^{-1,2}\left( G\right)$, then,
with probability one, equation \eqref{10} has a solution and
$$
\mathbb{P}\big(
\nu |\bld{u}|_{1,2}\leq
|\bld{f}|_{-1,2}\big)=1.
$$
If, in addition, there exists a non-random $\theta\in (0,1)$ such that
\begin{equation}
\label{small-f}
\mathbb{P}\left(
\left\vert \bld{f}\right\vert _{-1,2} \leq
\frac{2\theta \nu^{2}}{\sqrt{|G|}}
\right)=1,
\end{equation}%
then the solution is unique and satisfies
\begin{equation}
\label{small-u}
\mathbb{P}\left(
\left\vert \bld{u}\right\vert _{1,2} \leq
\frac{2\nu\theta}{\sqrt{|G|}}
\right)=1.
\end{equation}
\end{theorem}

Intuitively, once we know the velocity field
 $\bld{u}$, we should be able to
  recover pressure $\, p\, $ from the original
equation \eqref{10}. Lemma \ref{le2} confirms this intuition; see also
\cite[Proposition I.1.1]{temam}. As a result, in what follows, we only consider the function $\bld{u}$.

Sometimes it is convenient to work with alternative characterizations
of the solution of \eqref{10}.

\begin{proposition}
\label{p-sol-10}
Let $\mathbf{f\in }\mathbb{H}_{0}^{-1,2}\left( G\right) $,
 $\bld{u}\in \widehat{\mathbb{H}}_{0}^{1,2}\left( G\right)$,
 and let   $\mathfrak{a}$ be the tri-linear form \eqref{form-a}.
  Then $\bld{u}$ is
 a solution to \eqref{10} if and only of,
for every $\bld{w}\in \widehat{\mathbb{H}}_{0}^{1,2}\left( G\right)$,
\begin{equation}
\label{ff1-w}
\mathbb{P}\Big(\,\nu\, \big(\nabla \bld{u},\nabla \bld{w}\big)_0
 +
 \mathfrak{a}(\bld{u},\bld{u},\bld{w})
  =
  -\left\langle \bld{f},\bld{w }\right\rangle _{1}\Big) =1,
\end{equation}
or
\begin{equation}
\label{ff1-w-E}
\,\nu\,\mathbb{E} \big(\nabla \bld{u},\nabla \bld{w}\big)_0
 +
 \mathbb{E}\mathfrak{a}(\bld{u},\bld{u},\bld{w})
  =
  -\mathbb{E}\left\langle \bld{f},\bld{w }\right\rangle _{1}.
\end{equation}
\end{proposition}

\begin{proof}
By construction,
$$
\eqref{ff1}\ \Rightarrow\ \eqref{ff1-w}\ \Rightarrow\ \eqref{ff1-w-E}.
$$
To establish $ \eqref{ff1-w-E}\, \Rightarrow\, \eqref{ff1}$, take
$\bld{w}=\bld{\varphi}\, \zeta$ with $\bld{\varphi}\in \mathcal{D}(G)$
and a bounded  random variable $\zeta$.
\end{proof}

\begin{corollary}
Let  $\bld{f}, \bld{g} \in \mathbb{H}_{0}^{-1,2}\left( G\right)$ and
let  $\bld{u}, \bld{v}\in \widehat{\mathbb{H}}_{0}^{1,2}\left( G\right)$ be the
corresponding solutions of \eqref{10}. If \eqref{small-f} holds, then
 \begin{equation}
 \label{difference-0}
\mathbb{P}\left( | \bld{u}-\bld{v}|_{1,2}\leq \frac{|\bld{f}-\bld{g}|_{-1,2}}{\nu(1-\theta)}\right)=1.
 \end{equation}

\end{corollary}

\begin{proof}
By \eqref{ff1-w}, we have, with probability one,
 \begin{equation}
 \label{difference-1}
 \nu\, \big(\nabla (\bld{u}-\bld{v}),\nabla \bld{w}\big)_0
 +
 \mathfrak{a}(\bld{u},\bld{u},\bld{w})-\mathfrak{a}(\bld{v},\bld{v},\bld{w})
  =
  -\left\langle \bld{f}-\bld{g},\bld{w}\right\rangle _{1}.
\end{equation}
Taking $\bld{w}= \bld{u}-\bld{v}$ and using \eqref{cont-a-diff1}, we re-write \eqref{difference-1} as
$$
\nu| \bld{u}-\bld{v}|_{1,2}^2 - \frac{\sqrt{|G|}}{2}\,
|\bld{u}-\bld{v}|_{1,2}^2\,|\bld{u}|_{1,2}\leq \left\langle \bld{f}-\bld{g},\bld{u}-\bld{v}\right\rangle _{1}
\leq  |\bld{f}-\bld{g}|_{-1,2}\ |\bld{u}-\bld{v}|_{1,2},
$$
and then \eqref{difference-0} follows from \eqref{small-u}.

\end{proof}

\section{Analysis of the Stochastic Galerkin Approximation}
\label{sec2}

For an integer $N\geq 1$, consider the equation
\begin{eqnarray}
\label{30}
\nu \Delta \bld{v}_N &=&\mathcal{P}^{N}\Big((\bld{v}_N\cdot\nabla)\bld{v}_N\Big)
+\nabla p_N+
\mathcal{P}^{N}\bld{f},   \\
\mathrm{div}\,{}\bld{v}_N &=&0,\ \ \bld{v}_N\big|_{\partial G}=0.
\notag
\end{eqnarray}%
Recall that $\bld{f}$ is the random forcing of the form
 $$
 \bld{f}(x)=\mathbf{f}\left( \xi ,x\right) =
 \left( f^{1}\left(\xi ,x\right),f^{2}(\xi ,x)\right)
 $$
 with a  suitable (known) non-random  vector field $\mathbf{f}$ and a random variable $\xi$,
 and $\mathcal{P}^{N}$ is   the orthogonal projection in $L_2(\Omega ,\mathcal{F}_{\xi},\mathbb{P})$
   on the subspace spanned  by the first $N$ orthogonal polynomials corresponding to the distribution of $\xi$.

Similar to Definition \ref{d1}, we have
\begin{definition}
\label{d5}
Given  $\bld{f\in }\mathbb{H}_{0}^{-1,2}\left( G\right)$, a
random vector field
$\bld{v}_N \in \mathcal{P}^N\Big(\widehat{\mathbb{H}}_{0}^{1,2}
\left( G\right)\Big) $
 is called a solution of \eqref{30}, if, for every $\bld{\varphi}\in \mathcal{D}\left( G\right)$,
\begin{equation}
\label{ff7}
\mathbb{P}\Bigg(
\nu \, \big(\nabla \bld{v}_N,\nabla \bld{\varphi }\big)_0
 +
 \mathcal{P}^{N}\mathfrak{a}(\bld{v}_N,\bld{v}_N,\bld{\varphi })
 =-
 \left\langle \mathcal{P}^{N}%
\bld{f},\bld{\varphi }\right\rangle _{1}\Bigg)=1.
\end{equation}%
\end{definition}
We call $\bld{v}_N$
 a {\tt polynomial chaos approximation} of the solution
 $\bld{u}$ of equation \eqref{10}.

Similar to \eqref{ff1-w} and \eqref{ff1-w-E}, we will establish  two alternative  characterizations of the solution of
\eqref{30}.

If $\bld{u}\in \mathbb{H}_{0}^{1,2}
\left( G\right)$,  $\bld{v}\in \mathbb{H}_{0}^{1,2}
\left( G\right)$, and $\bld{w}\in  \mathcal{P}^N\Big(\mathbb{H}_{0}^{1,2}
\left( G\right)\Big)$, then
 $\mathcal{P}^N\bld{w}=\bld{w}$ and therefore
\begin{equation}
\label{eq:CP-pr}
\mathbb{E}\big(\mathcal{P}^N(\bld{u}\cdot\nabla)\bld{v},\bld{w}\big)
=\mathbb{E}\big((\bld{u}\cdot\nabla)\bld{v},\mathcal{P}^N\bld{w}\big)
=\mathbb{E}\,\mathfrak{a}(\bld{u},\bld{v},\bld{w}).
\end{equation}

As a result, direct computations lead to the first
 alternative characterization of the solution of \eqref{30}.
\begin{proposition}
\label{prop:equiv1}
A random vector field  $%
\bld{v}_N \in \mathcal{P}^N\Big(\widehat{\mathbb{H}}_{0}^{1,2}
\left( G\right)\Big) $
 is   a solution of \eqref{30} if and only if, for every
 $\bld{w}\in
 \mathcal{P}^N\Big(\widehat{\mathbb{H}}_{0}^{1,2}(G)\Big)$,
\begin{equation}
\label{NSE-Exp}
\nu\, \mathbb{E}\,
\big(\nabla \bld{v}_N,\nabla \bld{w}\big)
+
\mathbb{E}\,\mathfrak{a}(\bld{v}_N,\bld{v}_N,\bld{w})
=-
 \mathbb{E} \left\langle \bld{f},\bld{w}\right\rangle _{1}.
 \end{equation}
\end{proposition}
In particular, if a solution $\bld{v}_N$ exists, then, taking $\bld{w}= \bld{v}_N$ and using \eqref{CP}, we find
\begin{equation}
\label{vN-L2bnd}
\nu|\bld{v}_N|_{\mathbb{H}^{1,2}_0} \leq   |\bld{f}|_{\mathbb{H}^{-1,2}_0}.
\end{equation}

Equality \eqref{NSE-Exp} shows that  $\bld{v}_N$ is indeed  a
  {\tt stochastic Galerkin approximation}  of $\bld{u}$.

To derive yet another   form of \eqref{30}, start by writing
 $$
 \bld{v}_N=\sum_{l=0}^{N}\mathbf{v}^{l}_N\mathfrak{P}_{l},\qquad \mathcal{P}^N\bld{f}=\sum_{l=0}^N \mathbf{f}^l\mathfrak{P}_{l}.
 $$
Then,  using the numbers $A_{m,k;l}$ defined in  \eqref{eq:Amnk}, we compute
 \begin{equation}
 \label{conv-proj}
 \begin{split}
\Big((\bld{v}_N\cdot\nabla)\bld{v}_N\Big)
 &=\sum_{m,k=0}^{N}
  (\mathbf{v}^m_N\cdot\nabla)\mathbf{v}_N^k
  \mathfrak{P}_m\mathfrak{P}_k=
  \sum_{m,k=0}^{N}
  (\mathbf{v}^m_N\cdot\nabla)\mathbf{v}_N^k
  \sum_{l=0}^{m+k}A_{m,k;l}\mathfrak{P}_l\\
  &=
  \sum_{l=0}^{2N}
  \left(\sum_{m,n=0}^N A_{m,k;l}\,
(\mathbf{v}_N^m\cdot\nabla)
\mathbf{v}_N^k\right)\mathfrak{P}_l,
\end{split}
 \end{equation}
that is,
$$
\mathcal{P}^{N}\Big((\bld{v}_N\cdot\nabla)\bld{v}_N\Big)=
\sum_{l=0}^{N}
  \left(\sum_{m,n=0}^N A_{m,k;l}\,
(\mathbf{v}_N^m\cdot\nabla)
\mathbf{v}_N^k\right)\mathfrak{P}_l.
$$
  As a result,  \eqref{30} is equivalent to
 the following system of equations  for the non-random vector
   functions $\mathbf{v }^{l}_N$, $l=0,\ldots,N$:
\begin{equation}
\label{32}
\nu \Delta \mathbf{v}^{l}_N
=\sum_{m,n=0}^N A_{m,k;l}\,
(\mathbf{v}_N^k\cdot\nabla)
\mathbf{v}_N^m+\nabla p^{l}_N+\mathbf{f}^{l}.
\end{equation}
This system is more complicated than \eqref{10} and will require more sophisticated numerical procedures to compute a solution,
 whence the term ``intrusive'' in connection with  stochastic Galerkin approximation.

For example, if $\xi$ is a standard normal random variable and $\mathcal{F}=\sigma(\xi)$,  $\bld{f}(x)=\mathbf{f}\left( \xi ,x\right) =
 \left( f^{1}\left(\xi ,x\right),f^{2}(\xi ,x)\right)$
for a non-random  vector field $\mathbf{f}$, then $\mathfrak{P}_n=\mathrm{H}_n(\xi)$,
where
$$
\mathrm{H}_n(x)=(-1)^{n}e^{x^2/2}\frac{d^n e^{-x^2/2}}{dx^n}
$$
 is $n$-th Hermite polynomial, $c(n)=n!$, and
$$
\mathfrak{P}_m\mathfrak{P}_n=
\sum_{k=0}^{\min(m,n)}\frac{m!\,n!}{(m-k)!\,(n-k)!\,k!}\,
\mathfrak{P}_{m+n-2k}
$$
(cf. \cite[Formula (6.7)]{DXiu-book});
 after some algebraic manipulations,  \eqref{32} becomes
\begin{equation}
\label{32-G}
\nu \Delta \mathbf{v}_N^{l}=
\sum_{n=0}^{N}\frac{1}{n!}
\sum_{\substack{k+m=l,  \\ k+n\leq N,m+n\leq N}}
\frac{(k+n)!}{k!}\frac{(m+n)!}{m!}
\,(\mathbf{v}^{k+n}_N\cdot\nabla)\mathbf{v}^{m+n}_N
+\nabla p^{l}_N+\mathbf{f}^{l}.
\end{equation}

Combining \eqref{32} with Proposition \ref{prop:equiv1}, we get the
second alternative characterization of the solution of \eqref{30}.

\begin{proposition}
\label{prop:equiv2}
A collection of functions
$ \mathbf{v}^l_N,\ l=0,\ldots,N,$
  with each  $\mathbf{v}^l_N\in \widehat{\mathbf{H}}_{0}^{1,2}(G)$,
  is a solution of \eqref{32} if and only if,  for every collection of functions
 $\{\mathbf{w}^{l},\ l=0,\ldots,N\}$,
$\mathbf{w}^{l} \in \mathcal{D}\left( G\right)$, the following
equality holds:
\begin{equation}
\label{320}
\begin{split}
\nu \sum_{l=0}^{N}c(l)\,
\big(\nabla \mathbf{v}^{l}_N,\nabla \mathbf{w}^{l}\big)_0
&+\sum_{l=0}^N c(l)
\sum_{m,n=0}^N A_{m,k;l}\,
\mathfrak{a}(\mathbf{v}_N^k,\mathbf{v}_N^m,\mathbf{w}^l)\\
&=-
\sum_{l=0}^{N}c\left(
l\right) \left\langle \mathbf{f}^{l},\mathbf{w}^{l}\right\rangle _{1},\quad c(l)=\mathbb{E}\mathfrak{P}_l^2.
\end{split}
\end{equation}

\end{proposition}

Given $\bar{\mathbf{u}}$, $\bar{\mathbf{v}}$,
$\bar{\mathbf{w}}$ in $\big(\mathbf{H}^{1,2}_0(G)\big)^{N+1}$,
with $\bar{\mathbf{u}}=(\mathbf{u}^0,\ldots,\mathbf{u}^N)$
and similarly for $\bar{\mathbf{v}}, \bar{\mathbf{w}}$,
define
\begin{equation}
\label{A}
\mathfrak{A}(
\bar{\mathbf{u}},
\bar{\mathbf{v}},
\bar{\mathbf{w}})
=\sum_{l=0}^N c(l)
\sum_{m,n=0}^N A_{m,k;l}\,
\mathfrak{a}(\mathbf{u}^k,\mathbf{v}^m,\mathbf{w}^l).
\end{equation}
Then we can re-write \eqref{320} as
\begin{equation}
\label{320-s}
\nu \sum_{l=0}^{N}c(l)\, \big(\nabla\mathbf{v}^{l}_N,
\nabla \mathbf{w}^{l}\big)_0
+
\mathfrak{A}(
\bar{\mathbf{v}}_{{}_N},
\bar{\mathbf{v}}_{{}_N},
\bar{\mathbf{w}})
= -
\sum_{l=0}^{N}
c\left(l\right)
 \left\langle \mathbf{f}^{l},\mathbf{w}^{l}\right\rangle _{1}.
\end{equation}%

Furthermore, given $\bar{\mathbf{u}}$, $\bar{\mathbf{v}}$,
$\bar{\mathbf{w}}$ in $\big(\mathbf{H}^{1,2}_0(G)\big)^{N+1}$,
define
$$
{\bld{u}}=\sum_{k=0}^N \mathbf{u}^k\mathfrak{P}_k,\ \
{\bld{v}}=\sum_{k=0}^N \mathbf{v}^k\mathfrak{P}_k,\ \
{\bld{w}}=\sum_{k=0}^N \mathbf{w}^k\mathfrak{P}_k.
$$
Then equality \eqref{eq:CP-pr} implies
\begin{equation}
\label{A=Ea}
\mathfrak{A}(
\bar{\mathbf{u}},
\bar{\mathbf{v}},
\bar{\mathbf{w}})=\mathbb{E}\mathfrak{a}(\bld{u},\bld{v},\bld{w}).
\end{equation}
In particular, by \eqref{CP},
\begin{equation}
\label{Exp-cancel}
\mathfrak{A}(
\bar{\mathbf{u}},
\bar{\mathbf{v}},
\bar{\mathbf{v}})=0
\end{equation}
provided $\mathbf{u}^k\in \widehat{\mathbf{H}}^{1,2}_0(G)$
 for all $k=0,\ldots, N$.

We now use \eqref{320} to
establish a basic solvability result for equation \eqref{30}.
\begin{theorem}
\label{pro1}
For every  $\bld{f}\in\mathbb{H}_{0}^{-1,2}\left( G\right) $ and
 $N\geq 1$, equation \eqref{30} has a solution $\bld{v}_N$ and
\begin{equation}
 \label{9}
\left\vert \bld{v}_N\right\vert_{\mathbb{H}^{1,2}_0}  \leq
\frac{\left\vert
\bld{f}\right\vert _{\mathbb{H}^{-1,2}_0}}{\nu}.
\end{equation}
The solution is unique if there exists a non-random number $\varepsilon_N\in (0,1)$ such that
\begin{equation}
\label{Proj-US}
\mathbb{P}\Big( \vert \bld{v}_N\vert_{1,2}\leq
\frac{2\nu(1-\varepsilon_N)}{\sqrt{|G|}}\Big)=1.
\end{equation}
\end{theorem}

\begin{proof}
For  $M\geq 1$ and $l=0,\ldots, N$, define
\begin{equation}
\label{vMNl}
\mathbf{v}_{{}_{M,N}}^l=\sum_{k=0}^{M}z^{k,l}_{{}_{M,N}}
\mathbf{h}_k,
\end{equation}
where $z^{k,l}_{{}_{M,N}}\in \mathbb{R}$ and the
functions $\mathbf{h}_k$ have the following properties:
\begin{enumerate}
\item $\mathbf{h}_k\in \widehat{\mathbf{H}}^{1,2}_0(G)$, $k\geq 0$;
\item Finite linear combinations of $\mathbf{h}_k$ are dense in
the space $\widehat{\mathbf{H}}^{1,2}_0(G)$;
\item $|\mathbf{h}_k|_{\mathbf{L}^2}=1,\
(\mathbf{h}_k,\mathbf{h}_m)=0,\ k\not=m$.
\end{enumerate}
A possible choice is the normalized eigenfunctions of the Stokes operator \cite[Section I.2.6]{temam}.

Also, we will use the notations
$$
\bar{\mathbf{v}}_{{}_{M,N}}=
(\mathbf{v}_{{}_{M,N}}^0,\ldots, \mathbf{v}_{{}_{M,N}}^N),\ \ \
\bld{v}_{{}_{M,N}}=
\sum_{l=0}^N\mathbf{v}_{{}_{M,N}}^l\mathfrak{P}_l.
$$
Consider the system of equations
\begin{equation}
\label{GA1}
\nu \big(\nabla \mathbf{v}^{l}_{{}_{M,N}},\nabla\mathbf{h}_k\big)
+\sum_{m,n=0}^N A_{m,n;l}\,
\big((\mathbf{v}_{{}_{M,N}}^m\cdot\nabla)
\mathbf{v}_{{}_{M,N}}^n,\mathbf{h}_k\big)
+\langle\mathbf{f}^{l},\mathbf{h}_k\rangle_{1}=0,
\end{equation}
$k=0,\ldots, M,\ l=0,\ldots, N$; with in mind, we think of \eqref{GA1} as a system of equations  for the
numbers $z^{k,l}_{{}_{M,N}}$.

To show that \eqref{GA1} has a solution for every
$M\geq 1$, we introduce the following notations:
\begin{align*}
\bar{\mathbf{z}}&=\big(z^{0,0}_{{}_{M,N}}, \ldots, z^{M,0}_{{}_{M,N}},z^{0,1}_{{}_{M,N}}\ldots,
z^{M,1}_{{}_{M,N}},
\ldots,
z^{0,N}_{{}_{M,N}},\ldots,z^{M,N}_{{}_{M,N}}\big),\\
Q_{k,l}(\bar{\mathbf{z}})&=
\nu \big(\nabla \mathbf{v}^{l}_{{}_{M,N}},\nabla\mathbf{h}_k\big)
+\sum_{m,n=0}^N A_{m,n;l}\,
\big((\mathbf{v}_{{}_{M,N}}^m\cdot\nabla)
\mathbf{v}_{{}_{M,N}}^n,\mathbf{h}_k\big)
+\langle\mathbf{f}^{l},\mathbf{h}_k\rangle_{1},\\
F(\bar{\mathbf{z}})&=
\sum_{k,l} c(l)Q_{k,l}(\bar{\mathbf{z}})z^{k,l}_{{}_{M,N}}.
\end{align*}
Combining  \eqref{vMNl},   \eqref{320-s},  and \eqref{Exp-cancel},
\begin{equation}
\label{ES-1}
\begin{split}
F(\bar{\mathbf{z}})&=
\nu \sum_{l=0}^{N}c(l)\,|\mathbf{v}^{l}_{{}_{M,N}}|_{1,2}^2
+
\mathfrak{A}(
\bar{\mathbf{v}}_{{}_{M,N}},
\bar{\mathbf{v}}_{{}_{M,N}},
\bar{\mathbf{v}}_{{}_{M,N}})
+
\sum_{l=0}^{N}
c\left(l\right)
 \left\langle \mathbf{f}^{l},\mathbf{v}^{l}_{{}_{M,N}}\right\rangle _{1}
 \\
 &=
 \nu \sum_{l=0}^{N}c(l)\,|\mathbf{v}^{l}_{{}_{M,N}}|_{1,2}^2
+
\sum_{l=0}^{N}
c\left(l\right)
 \left\langle \mathbf{f}^{l},\mathbf{v}^{l}_{{}_{M,N}}\right\rangle _{1}.
\end{split}
\end{equation}
It follows that
$$
F(\bar{\mathbf{z}})\geq 0
$$
if
$$
\frac{2\nu^2}{|G|} \sum_{k,l} c(l)\big|z^{k,l}_{{}_{M,N}}\big|^2=
|\bld{f}|^2_{\mathbb{H}_0^{-1,2}}.
$$
Indeed, by the Cauchy-Schwarz inequality,
$$
F(\bar{\mathbf{z}})\geq
|\bld{v}_{{}_{M,N}}|_{\mathbb{H}_0^{1,2}}
\Big(\nu |\bld{v}_{{}_{M,N}}|_{\mathbb{H}_0^{1,2}} -
|\bld{f}|_{\mathbb{H}_0^{-1,2}}\Big),
$$
whereas the  Poincar\'{e} inequality \eqref{PI} implies
$$
|\bld{v}_{{}_{M,N}}|_{\mathbb{H}_0^{1,2}}^2
\geq \frac{2|\bld{v}_{{}_{M,N}}|_{\mathbb{L}^{2}}^2}{|G|}=
\frac{2}{|G|} \sum_{k,l} c(l)\big|z^{k,l}_{{}_{M,N}}\big|^2.
$$
By a multi-dimensional version of the intermediate value
theorem \cite[Lemma IX.3.1]{ga}, we conclude that there exists a
$\bar{\mathbf{z}}^*$ with
$$
\frac{2\nu^2}{|G|}
 \sum_{k,l} c(l)\big|z^{*,k,l}_{{}_{M,N}}\big|^2
 \leq
|\bld{f}|^2_{\mathbb{H}_0^{-1,2}}
$$
such that $Q_{k,l}(\bar{\mathbf{z}}^*)=0$ for all $k,l$.
In other words, we now have existence of solution of
\eqref{GA1} for every $M\geq 1$. Moreover,  by \eqref{ES-1}
and the Cauchy-Schwarz inequality, the solution satisfies
$$
\left\vert \bld{v}_{{}_{M,N}}\right\vert_{\mathbb{H}^{1,2}_0} \leq
\frac{\left\vert
\bld{f}\right\vert _{\mathbb{H}^{-1,2}_0}}{\nu},
$$
which, together with the compactness of the embedding
$\mathbf{H}^{1,2}_0(G)\subset \mathbf{L}^2(G)$, allows us
to pass to the limit  $M\to \infty$ and get a solution of \eqref{320}
in the same way as in \cite[pp. 600--601]{ga}.

To establish uniqueness, let $\bld{v}_N$ be the  solution
of \eqref{320} satisfying
\eqref{Proj-US} and let $\tilde{\bld{v}} $ the difference between
$\bld{v}_N$ and any other possible solution of \eqref{320}.
By \eqref{NSE-Exp},
\begin{equation}
\label{US-GA1}
\nu\, \mathbb{E}(\nabla\tilde{\bld{v}},\nabla\bld{w})_0+
 \mathbb{E}\mathfrak{a}(\bld{v}_N,\tilde{\bld{v}},\bld{w})-
 \mathbb{E}
 \mathfrak{a}(\tilde{\bld{v}},\tilde{\bld{v}}-\bld{v}_N,\bld{w})=0.
 \end{equation}
Because $\mathcal{P}^N\tilde{\bld{v}}=\tilde{\bld{v}}$ and
$\tilde{\bld{v}}\in \widehat{\mathbb{H}}^{1,2}_0(G)$,
we can put $\bld{w}=\tilde{\bld{v}}$ in \eqref{US-GA1} and then
use \eqref{CP} and \eqref{AS} to conclude that
\begin{equation*}
\nu\, \mathbb{E}|\tilde{\bld{v}}|_{1,2}^2+
 \mathbb{E}
 \mathfrak{a}(\tilde{\bld{v}},\bld{v}_N,\tilde{\bld{v}})=0,
 \end{equation*}
 which, together with \eqref{cont-a0} implies
 \begin{equation}
 \label{US-GA2}
 \mathbb{E}\left(|\tilde{\bld{v}}|_{1,2}^2
 \left(\nu-\frac{\sqrt{|G|}}{2}|\bld{v}_N|_{1,2}\right)\right)\leq 0.
 \end{equation}
 If \eqref{Proj-US} holds, then
 $$
 \nu-\frac{\sqrt{|G|}}{2}|\bld{v}_N|_{1,2}\geq \nu \varepsilon_N>0,
 $$
 and \eqref{US-GA2} is only possible when
 $ \mathbb{P} \big(|\tilde{\bld{v}}|_{1,2}^2=0\big)=1$, that is, when
 $\bld{v}_N$ is the unique solution of \eqref{320}.
\end{proof}

Similar to \eqref{small-f}, we need \eqref{Proj-US} to guarantee uniqueness of the stochastic Glerkin approximation. In fact,
 without \eqref{small-f}, uniqueness can fail for the original equation \eqref{10}  \cite[Section IX.2]{ga}.
Even though system of equations  \eqref{320} has been successfully used for numerical simulations \cite{StatStochNSE1, StatStochNSE2}, it is not immediately clear how condition
 \eqref{Proj-US} can be verified.

The following theorem is the first key result of the paper and  shows that, under \eqref{small-f} and \eqref{Proj-US}, stochastic Galerkin approximation
 $\bld{v}_N$ is indeed an approximation of the orthogonal projection $\mathcal{P}^N\bld{u}$.
 In particular, if $\varepsilon_N$ does not depend on $N$, then
   stochastic Galerkin approximation is asymptotically equivalent to the
orthogonal projection, in the sense that, as $N\to \infty$,  both converge to the true solution at the same rate.

\begin{theorem}
\label{t2}
Assume that \eqref{small-f} holds so that \eqref{10} has a unique
 solution $\bld{u}$, and let
$\bld{v}_N=\bld{v}_N(\xi)$ be the unique solution of \eqref{30} satisfying
\eqref{Proj-US}.
Then
\begin{align}
\label{lim-p}
&|\mathcal{P}^N\bld{u}-\bld{v}_N|_{\mathbb{H}^{1,2}_0}
\leq \frac{\theta+1-\varepsilon_N}{\varepsilon_N}\,
|\bld{u}-\mathcal{P}^N\bld{u}|_{\mathbb{H}^{1,2}_0},\\
\label{lim-intr}
&| \bld{u}-\bld{v}_N|_{\mathbb{H}^{1,2}_0}
\leq \left(1+\frac{\theta+1-\varepsilon_N}{\varepsilon_N}\right)\,
|\bld{u}-\mathcal{P}^N\bld{u}|_{\mathbb{H}^{1,2}_0}.
\end{align}
\end{theorem}

\begin{proof}
To make the formulas shorter, we write
$$
\bld{u}^N=\mathcal{P}^N\bld{u},\
 \bld{u}_N=\bld{u}^N-\bld{v}_N, \ \ {\rm and}\ \
 \bld{w}^N=\mathcal{P}^N\bld{w} \ \ {\rm for}\ \ \bld{w}\in\widehat{\mathbb{H}}_{0}^{1,2}(G)
$$
Using \eqref{ff1-w-E},
$$
\nu\,\mathbb{E}\big(\nabla \bld{u}^N,\nabla\bld{w}^N\big)_0+
\mathbb{E}\mathfrak{a}(\bld{u},\bld{u},\bld{w}^N)=
-\mathbb{E}\langle \bld{f},\bld{w}^N\rangle_1,
$$
and, after subtracting \eqref{NSE-Exp},
\begin{equation*}
\nu\,\mathbb{E}\big(\nabla \bld{u}_N,\nabla\bld{w}^N\big)_0+
\mathbb{E}\mathfrak{a}(\bld{u},\bld{u},\bld{w}^N)-
\mathbb{E}\mathfrak{a}(\bld{v}_N,\bld{v}_N,\bld{w}^N)=0.
\end{equation*}
Next,
\begin{align*}
\mathfrak{a}(\bld{u},\bld{u},\bld{w}^N)-
\mathfrak{a}(\bld{v}_N,\bld{v}_N,\bld{w}^N)&=
\mathfrak{a}(\bld{u},\bld{u}-\bld{v}_N,\bld{w}^N)+
\mathfrak{a}(\bld{u}-\bld{v}_N,\bld{v}_N,\bld{w}^N)\\
& =
\mathfrak{a}(\bld{u},\bld{u}-\bld{u}^N,\bld{w}^N)+
\mathfrak{a}(\bld{u},\bld{u}^N-\bld{v}_N,\bld{w}^N)\\
&+
\mathfrak{a}(\bld{u}-\bld{u}^N,\bld{v}_N,\bld{w}^N)+
\mathfrak{a}(\bld{u}^N-\bld{v}_N,\bld{v}_N,\bld{w}^N).
\end{align*}
Taking $\bld{w}^N=\bld{u}_N$ leads to
$$
\nu\,\mathbb{E}|\bld{u}_N|_{1,2}^2
+\mathbb{E}\mathfrak{a}(\bld{u}_N,\bld{v}_N,\bld{u}_N)
+\mathbb{E}\mathfrak{a}(\bld{u},\bld{u}-\bld{u}^N,\bld{u}_N)
+\mathbb{E}\mathfrak{a}(\bld{u}-\bld{u}^N,\bld{v}_N,\bld{u}_N)
=0.
$$
Then \eqref{cont-a0}, \eqref{small-u}, \eqref{Proj-US}, and the
Cauchy-Schwarz inequality [for expectations]  imply
$$
\nu\varepsilon_N  |\bld{u}_N|_{\mathbb{H}_0^{1,2}}^2
\leq \nu(\theta+1-\varepsilon_N) |\bld{u}_N|_{\mathbb{H}_0^{1,2}}\
|\bld{u}-\bld{u}^N|_{\mathbb{H}_0^{1,2}}.
 $$
We now get \eqref{lim-p}, and then, by triangle inequality, \eqref{lim-intr}.
\end{proof}

\section{A Non-Intrusive Approximation Using Gauss Quadrature}
\label{sec:GQ}

Let $\left( \Omega ,\mathcal{F},\mathbb{P}\right) $
be a probability space with a random variable $\xi$
 and let $\mathcal{F}_{\xi}$ be the $\mathbb{P}$-completion of
the sigma algebra generated by $\xi$. We assume that
 the moment generating function
 $
 \lambda\mapsto \mathbb{E}e^{\lambda\xi}
 $
 is defined in some neighborhood of $\lambda=0$.
 Under this  assumption, given a collection
 $\{P_n,\ n\geq 0\}$  of orthogonal polynomials corresponding to the distribution of $\xi$, the collection of random variables
 $$
 \mathfrak{P}_{n}=P_{n}\left( \xi\right),\ n\geq 0,
 $$
 is an orthogonal basis  in
 $L_2(\Omega ,\mathcal{F}_{\xi},\mathbb{P})$.
  Denote by  $\mathcal{P}^{N}$
  the orthogonal projection
   in $L_2(\Omega ,\mathcal{F}_{\xi},\mathbb{P})$
   on the subspace spanned  by
$\left\{ \mathfrak{P}_{k},\ k=0,\ldots ,N\right\}.$
Let
$$
c(n)=\mathbb{E}\mathfrak{P}_{n}^2,
$$
so that, for every $\zeta\in L_2(\Omega ,\mathcal{F}_{\xi},\mathbb{P})$,
$$
\zeta=\sum_{k\geq 0}
\frac{\mathbb{E}\big(\zeta\mathfrak{P}_k)}{c(k)}\,
\mathfrak{P}_{k}.
$$

In this section, we assume that the random forcing in equation \eqref{10}  has a special form
\begin{equation}
\label{f-1d}
 \bld{f}(x)=\mathbf{f}\left( \xi ,x\right) =
 \left( f^{1}\left(\xi ,x\right),f^{2}(\xi ,x)\right),
 \end{equation}
 where $\mathbf{f}$ is a   non-random  vector field.
 If $\bld{u}=\bld{u}(\xi)$ is a solution of \eqref{10} corresponding to the particular realization of $\xi$ and $\bld{u}^N=\mathcal{P}^{N}\bld{u}$, then
$$
\bld{u}(\xi) \approx \bld{u}^N(\xi),  \quad \bld{u}^N(\xi) = \sum_{k=0}^{N}\mathbf{u}_k \frac{\mathfrak{P}(\xi)}{c(k)},
\quad \mathbf{u}_k=\bE\big(\bld{u}\mathfrak{P}_k\big).
$$
To compute  the coefficients $\mathbf{u}_k,\ k=0,\ldots,N$, we use the {\em Gauss quadrature} approximation
$ \mathbf{u}_k\approx  \mathbf{u}^{(N)}_k$, where
\begin{equation}
\label{eq:NI0}
 \mathbf{u}^{(N)}_k= \sum_{j=1}^{N+1} w_{j,N}\bld{u}(\xi_{j,N})\mathfrak{P}_k(\xi_{j,N}),
\end{equation}
  $\xi_{j,N},\ j=1,\ldots, N+1,$ are the roots of $P_{N+1}$, and $w_{j,N}$ are the corresponding weights; cf. \cite[Section 1.4]{WGaut-OrtPoly}.
  The resulting {\em discrete projection} or {\em pseudo-spectral} approximation,
\begin{equation}
\label{eq:NI-1}
\bld{u}^{(N)}(\xi)= \sum_{k=0}^N \mathbf{u}^{(N)}_k\, \frac{\mathfrak{P}_k(\xi)}{c(k)},
\end{equation}
requires the solution $\bld{u}(\xi_{j,N})$ of \eqref{10} for $N+1$ distinct values of $\xi$.

To simplify the formulas, it is convenient to introduce the square matrix $\mathfrak{W}= \big( \mathfrak{W}_{kj},\ k=0,\ldots,N, \ j=1,\ldots, N+1\big)$, with
$$
\mathfrak{W}_{kj}=w_{j,N}\mathfrak{P}_k(\xi_{j,N}).
$$
Then \eqref{eq:NI0} becomes
\begin{equation}
\label{eq:NI0-w}
 \mathbf{u}^{(N)}_k= \sum_{j=1}^{N+1}\mathfrak{W}_{kj}\bld{u}(\xi_{j,N}).
\end{equation}

The basic property  of the Gauss quadrature is that the equality
$$
\bE h(\xi)=\sum_{j=1}^{N+1} w_{j,N} h(\xi_{j,N})
$$
holds for all functions $h=h(\xi)$ that are polynomials in $\xi$ of degree at most  $2N+1$; cf. \cite[Theorem 1.45]{WGaut-OrtPoly}.
In particular, for every $\ k,m=0,\ldots, N,$
$$
\sum_{j=1}^{N+1}\mathfrak{W}_{kj}\mathfrak{P}_m(\xi_{j,N}) = \sum_{j=1}^{N+1} w_{j,N}\mathfrak{P}_k(\xi_{j,N})
\mathfrak{P}_m(\xi_{j,N})=\bE\big(\mathfrak{P}_k\mathfrak{P}_m\big)=
\begin{cases}
c(k)>0,& \text {if } k=m,\\
0,& \text {if } k\not=m,
\end{cases}
$$
which means that the matrix $\mathfrak{W}$ is non-singular.

With the above choice of the sampling points $\xi_{j,N}$, the discrete projection \eqref{eq:NI-1} is equivalent to interpolation:

\begin{proposition}
The equality
\begin{equation}
\label{eq:NI-2}
\bld{u}(\xi_{j,N})=\bld{u}^{(N)}(\xi_{j,N})
\end{equation}
holds for all $j=1,\ldots, N+1$.
\end{proposition}

\begin{proof}
Equality \eqref{eq:NI-1} implies that  $\bld{u}^{(N)}$ is a polynomial in $\xi$ of order at most $N$, so that
each product $\bld{u}^{(N)}\mathfrak{P}_k$, $k=0,\ldots, N,$ is a polynomial in $\xi$ or order at most $2N$.
Then
\begin{equation}
\label{eq:NI-2-1}
\bE \big(\bld{u}^{(N)}\mathfrak{P}_k\big) = \sum_{j=1}^{N+1} \mathfrak{W}_{kj}\bld{u}^{(N)}(\xi_{j,N}),\ k=0,\ldots, N.
\end{equation}
On the other hand, \eqref{eq:NI-1} also implies
\begin{equation}
\label{eq:NI-2-2}
\bE \big(\bld{u}^{(N)}\mathfrak{P}_k\big)= \mathbf{u}^{(N)}_k,
\end{equation}
and then \eqref{eq:NI-2} follows from \eqref{eq:NI0-w} and non-degeneracy of the matrix $\mathfrak{W}$.
\end{proof}

The following theorem is the second key result of the paper and  gives an upper bound on the approximation error $\bE |\bld{u} -\bld{u}^{(N)} |^2_{1,2}$.
Recall that $\bld{u}^N=\mathcal{P}^N\bld{u}$.

\begin{theorem}
\label{th:NI}
Define
\begin{equation}
\label{sup-er-NI}
\delta_N=\sup_{\xi} |\bld{u}(\xi)-\bld{u}^N(\xi)|_{1,2}.
\end{equation}
Then
\begin{equation}
\label{eq:NI-main}
\bE |\bld{u} -\bld{u}^{(N)} |^2_{1,2}\leq \bE|\bld{u} -\bld{u}^N |_{1,2}^2+N\big(\delta_N\big)^2.
\end{equation}
\end{theorem}

\begin{proof}
By orthogonality,
\begin{equation}
\label{eq:NI-main-1}
\begin{split}
\bE |\bld{u} -\bld{u}^{(N)} |^2_{1,2}&= \bE|\bld{u} -\bld{u}^N |^2_{1,2} +
\bE| \bld{u}^N - \bld{u}^{(N)} |^2_{1,2}\\
& = \bE|\bld{u} -\bld{u}^N |^2_{1,2} +
\sum_{k=0}^N \frac{|\mathbf{u}_k-\mathbf{u}_k^{(N)}|_{1,2}^2}{c(k)}.
\end{split}
\end{equation}
Combining \eqref{eq:NI0}, \eqref{eq:NI-2-1}, and \eqref{eq:NI-2-2}  results in
$$
 \mathbf{u}_k-\mathbf{u}_k^{(N)} = \sum_{j=1}^{N+1} w_{j,N} \big(\bld{u}^N(\xi_{j,N})-\bld{u}(\xi_{j,N})\big)\mathfrak{P}_k(\xi_{j,N})
 $$
 or, using the Cauchy-Schwarz inequality and $w_{j,N}>0$,
 $$
 |\mathbf{u}_k-\mathbf{u}_k^{(N)}|_{1,2}^2 \leq
  \left(\sum_{j=1}^{N+1} w_{j,N} |\bld{u}^N(\xi_{j,N})-\bld{u}(\xi_{j,N})|_{1,2}^2 \right)\left(
  \sum_{j=1}^{N+1}w_{j,N} \mathfrak{P}^2_k(\xi_{j,N})\right).
  $$
  Properties of the Gauss quadrature imply
  $$
  \sum_{j=1}^{N+1}w_{j,N} \mathfrak{P}^2_k(\xi_{j,N})= \bE\mathfrak{P}^2_k=c(k),\ k=0,\ldots, N,
  \text{\ and \ }
   \sum_{j=1}^{N+1}w_{j,N} =1,
  $$
whereas  \eqref{sup-er-NI} implies
 $$
 \sum_{j=1}^{N+1} w_{j,N} |\bld{u}^N(\xi_{j,N})-\bld{u}(\xi_{j,N})|_{1,2}^2 \leq \big(\delta_N\big)^2 \sum_{j=1}^{N+1}w_{j,N},
 $$
As a result,
$$
|\mathbf{u}_k-\mathbf{u}_k^{(N)}|_{1,2}^2\leq \big(\delta_N\big)^2c(k),
$$
and  \eqref{eq:NI-main} follows from \eqref{eq:NI-main-1}.
\end{proof}

Of course, $\bE|\bld{u} -\bld{u}^N |_{1,2}^2\leq \big(\delta_N\big)^2$, leading to a somewhat weaker form of \eqref{eq:NI-main}:\\
$
\bE |\bld{u} -\bld{u}^{(N)} |^2_{1,2} \leq \big(\delta_N\big)^2 (1+N).
$

\begin{remark} {\rm
Both intrusive and non-intrusive approximations require  an $L_{\infty}$-bound, either   in the form of  \eqref{Proj-US} or \eqref{sup-er-NI},   to
establish an $L_2$-bound on the approximation error; for  \eqref{eq:NI-main} to be useful, one additionally needs to establish
\begin{equation}
\label{eq:NI-main-A}
\lim_{N\to \infty} \sqrt{N}\delta_N =0.
\end{equation}
On the one hand, condition \eqref{eq:NI-main-A} is easier to verify   than condition  \eqref{Proj-US}. On the other hand, under condition  \eqref{Proj-US},
the  error bound \eqref{lim-intr}  can be better than \eqref{eq:NI-main}, and this difference can become even more pronounced as the
stochastic dimension of the problem (the number of independent  random variables in the input) grows.
}
\end{remark}

The proof of Theorem \ref{th:NI} does not use the fact that $\bld{u}$ solves \eqref{10}. This additional information about $\bld{u}$, as well as the properties of the random
variable $\xi$ and the function $\bld{f}(x)=\mathbf{f}(\xi,x)$, are necessary  to establish \eqref{eq:NI-main-A}.

As an example, consider the random variable $\xi$ that is uniform on $[-1,1]$.
Then $\mathfrak{P}_n=P_n(\xi)$, where  $P_n$ is  $n$th Legendre polynomial; the standard normalization \cite[equation (6.4.4.)]{SpFunct-AAR} is $P_n(1)=1$, and then
$$
c_n=\frac{1}{2}\int_{-1}^1 P_n^2(x)\, dx = \frac{1}{2n+1}.
$$

\begin{theorem}
\label{th:IN-Leg}
Assume that, in \eqref{f-1d}, the random variable $\xi$ is uniform on $[-1,1]$ and the function $\mathbf{f}$ is Lipschitz continuous as a function of  $\xi$:
there exists a positive number $C_f$ such that, for all $\xi_1,\xi_2\in [-1,1]$,
\begin{equation}
\label{Lip-f}
|\mathbf{f}(\xi_1,\cdot)-\mathbf{f}(\xi_2,\cdot)|_{-1,2}\leq C_f|\xi_1-\xi_2|.
\end{equation}
 If  \eqref{small-f} holds and $\bld{u}=\bld{u}(\xi)$ is the corresponding unique solution of \eqref{10}, then
 \begin{equation}
 \label{eq:IN-Leg}
 \sup_{\xi} |\bld{u}(\xi)-\bld{u}^N(\xi)|_{1,2}\leq CN^{-3/4}
 \end{equation}
 for some $C$ depending only on $C_f, \nu$, and $\theta$.
 In particular, we have \eqref{eq:NI-main-A}.
 \end{theorem}

 \begin{proof} By  \eqref{difference-0},
 \begin{equation}
 \label{Lip-sol}
 |\bld{u}(\xi_1)-\bld{u}(\xi_2)|_{1,2}\leq \frac{|\mathbf{f}(\xi_1,\cdot)-\mathbf{f}(\xi_2,\cdot)|_{-1,2}}{\nu(1-\theta)} \leq
 \frac{C_f}{\nu(1-\theta)}\, |\xi_1-\xi_2|.
 \end{equation}
 For the rest of the proof, $C$ denotes  positive number depending only on $C_f, \nu$, and $\theta$. The value of $C$ can be different in different formulas.

 Let  $\mathcal{E}_N$ be the error of the best uniform approximation of $\bld{u}$ by an element of $\widehat{\mathbb{H}}^{1,2}_{0}(G)$ that is
 a  polynomial of degree at most $N$ in $\xi$:
 $$
 \mathcal{E}_N(\bld{u})=\inf\Big(\max_{\xi \in [-1,1]}|\bld{u}(\xi)-\bld{v}(\xi)|_{1,2} : \bld{v}\in \mathcal{P}^N\big(\widehat{\mathbb{H}}^{1,2}_{0}(G)\big)\Big).
 $$
 Then
 \begin{itemize}
 \item Jackson's Theorem \cite[Theorem 1.4]{Rivlin-appr}, together with \eqref{Lip-sol}, implies
 \begin{equation}
 \label{UBn1}
  \mathcal{E}_N(\bld{u})\leq \frac{C}{N};
  \end{equation}
  \item Combining \eqref{Lip-sol} with  \cite[Theorem 2.1]{Legendre-Appr1} yields
  \begin{equation}
  \label{UBn2}
  \bE |\bld{u} -\bld{u}^{(N)} |^2_{1,2} \leq \frac{C}{N^3};
  \end{equation}
  \item Combining   \eqref{UBn1} and \eqref{UBn2} with \cite[ Theorem 1 (p=2)]{Legendre-Appr3} leads to
     \eqref{eq:IN-Leg} and completes  the proof.
  \end{itemize}

\end{proof}

\section{Summary and Discussion}
\label{sec4}

Within the general framework of numerical analysis, this paper studies {\em a priori} error bounds, as opposed to {\em a posteriori} error analysis that requires some basic knowledge
about convergence of the numerical procedure; cf. \cite[Section 9.3]{ErrBnd-AO}.
Comparing the (intrusive) stochastic Galerkin approximation [Theorem \ref{t2}] and a (non-intrusive) stochastic collocation/Gauss quadrature approximation [Theorem \ref{th:NI}]
for  equation \eqref{10}, we see that
\begin{itemize}
\item The intrusive approximation  works for a broader class of random input and can, in principle,  achieve an asymptotically optimal rate of convergence;
\item The non-intrusive approximation is easier to study, both analytically and numerically.
\end{itemize}

The main technical difficulties to overcome   when analyzing
stochastic Galerkin approximation in general and when proving
  Theorem \ref{t2} in particular is related to the
fact that, for a  nonlinear equation,
$$
\mathcal{P}^N\bld{v}\not= \bld{v}_N.
$$
The  two possible  sources of non-linearity are
(a) the structure of the underlying deterministic equation, and (b)
the way the random perturbation enters the equation. For example,
the heat equation
\begin{equation}
\label{eq-S1}
\frac{\partial \bld{v}}{\partial t}=a\, \Delta \bld{v}
\end{equation}
with random $a$   is non-linear when it comes to
  polynomial chaos approximation. Another example is \eqref{10}
  with random $\nu$; see  \eqref{32-nu} below.
  This difficulty can be somewhat mitigated by replacing the usual product with
Wick product  \cite{KL-BE, mr12}, which is a
 convolution operation $\diamond$ such that
$\mathfrak{P}_m\diamond\mathfrak{P}_n=
\alpha_{mn}\,\mathfrak{P}_{m+n}$,
$\alpha_{mn}\in \mathbb{R}$ \cite{mr01}; the price
to pay is reduction of  physical relevance of the resulting equations.
Moreover, the analysis is much more manageable for equations of the type \eqref{eq-S1},
 when the underlying deterministic equation is linear \cite{StochGalerkin-error3, StochGalerkin-error4}.
 In fact, it is the ``deterministic nonlinearity'' that leads to  hard-to-verify conditions of the type \eqref{Proj-US}.

While the setting in the paper, a stationary two-dimensional Navier-Stokes
system with zero boundary conditions and additive random perturbation,
is intentionally simple to isolate the effects of non-linearity
(the convection term)  on the stochastic Galerkin approximation, some of the results are rather
universal and can be used for many other equations with a quadratic-type
nonlinearity. The key is equality  \eqref{conv-proj} describing the
product of two  chaos expansions.

For example, \eqref{conv-proj} implies that
equations \eqref{32} describe the stochastic Galerkin approximation
for the stationary Navier-Stokes system in any number of dimensions and
with randomness in both boundary conditions and the external force;
 after minor modifications, time-dependent problems
 with a random initial condition will also be covered. The number of random
 variables does not matter either, as long as the corresponding
 orthogonal basis $\{\mathfrak{P}_n,\ n\geq 0\}$ can be constructed
 \cite{mr01}.

Similarly, \eqref{conv-proj} shows that the stochastic Galerkin approximation for the system \eqref{10} with random viscosity will be
\begin{equation}
\label{32-nu}
\sum_{m,k=0}^N A_{m,k;l}\,\nu_k \,\Delta \mathbf{v}^{m}_N
=\sum_{m,k=0}^N A_{m,k;l}\,
(\mathbf{v}_N^k\cdot\nabla)
\mathbf{v}_N^m+\nabla p^{l}_N+\mathbf{f}^l,\ l=0,\ldots,N,
\end{equation}
where we assume
$$
\nu=\sum_{k=0}^{\infty} \nu_k\mathfrak{P}_k.
$$
Equation \eqref{32-nu} illustrates the effects of  two sources of
 nonlinearity: the convection term leads to the coupling of the functions  $\mathbf{v}_N^m$ on the right-hand side, whereas random viscosity
leads  to a similar  coupling  on the left-hand side. While not very different from \eqref{32},
analysis of \eqref{32-nu} must be carried out from scratch and, for now,
 is left to an interested reader.

To conclude, let us note that there are many equivalent ways to write
Navier-Stokes equations: even the basic velocity-pressure formulation
\eqref{10} admits at least four alternative forms
\cite[Section 5]{NSE-forms},
not to mention alternative variables, such stream  function and
vorticity \cite{Vorticity-Stationary, Vorticity-basic}.
 For the purpose of our investigation, it appears that
none of the alternatives will lead to any major simplifications, but,
as reference \cite{NSE-forms} suggests,
 one should keep those alternatives in
mind for further analysis of various approximations of \eqref{10}.


\providecommand{\bysame}{\leavevmode\hbox to3em{\hrulefill}\thinspace}
\providecommand{\MR}{\relax\ifhmode\unskip\space\fi MR }
\providecommand{\MRhref}[2]{%
  \href{http://www.ams.org/mathscinet-getitem?mr=#1}{#2}
}
\providecommand{\href}[2]{#2}

\end{document}